\documentclass[reqno]{amsart}
\begin{document}
\title[limiting behavior]
{Limiting behavior of solutions for Euler equations of compressible fluid flow}

\author[ Sahoo and Sen]
{ Manas Ranjan Sahoo and Abhrojyoti Sen}

\address{ Manas Ranjan Sahoo and Abhrojyoti Sen \newline
School of Mathematical Sciences\\
National Institute of Science Education and Research, HBNI, Bhubaneswar\\
Jatni, Bhimpur- Padanpur, Khurda\\
Odisha-752050, India}
\email{manas@niser.ac.in, abhrojyoti.sen@niser.ac.in}

\thanks{Submitted}

\begin{abstract}
  We study the limiting behavior of the solutions of  Euler equations of one-dimensional compressible fluid flow as the pressure like term vanishes. This system can be thought of as an approximation for the one dimensional model for large scale structure formation of universe.  We show that the solutions of former equation converges to the solution of later in the sense of distribution and agrees with the vanishing viscosity limit when the initial data is of Riemann type. A different approximation for the one dimensional model for large scale structure formation of universe is also studied.
\end{abstract}

\maketitle
\numberwithin{equation}{section}
\numberwithin{equation}{section}
\newtheorem{theorem}{Theorem}[section]
\newtheorem{remark}[theorem]{Remark}
\newtheorem{lem}[theorem]{Lemma}

\section{Introduction}
 This article is an attempt to establish a connection between the solutions of  two well known equations. One of them is called as Euler equation of one-dimensional compressible fluids, which is an example of strictly hyperbolic system where as the other one is a non strictly hyperbolic system, called one dimensional equation of large scale structure formation of universe.\\
 Euler equation of one-dimensional compressible fluid flow reads
\begin{equation}
\begin{aligned}
u_t +(\frac{u^2}{2}+ P(\rho))_x &=0\\
\rho_t +(\rho u)_x&=0,
\end{aligned}
\label{e1.3}
\end{equation}
with the initial condition

\begin{equation}
u(x,0)=u_{0}(x),\,\, \rho(x,0)=\rho_0 (x).
\label{e1.2}
\end{equation}
The equation \eqref{e1.3} was first derived by S. Earnshaw \cite{Ea, Wh} for isentropic flow. It is a scaling limit system of a Newtonian dynamics with 
long range interaction for a continuous distribution of mass \cite{Oe1, Oe2}. This equation is also hydrodynamic limit of Vlasov equation \cite{Cemp}.\\
We take
\begin{equation*} 
P(\rho)=\int_{0}^{\rho}{\frac{q^{'}(\xi)}{\xi}} d\xi
\end{equation*}
and intend to do all the analysis when $q$ is defined in the form 
\begin{equation*}
q(\rho)=\int_{0}^{\rho}{\xi}^2 \exp(\xi) d\xi.
\end{equation*}
The  existence viscosity solution of \eqref{e1.3} with initial data $\rho_0(x) > 0$  by parabolic regularization was shown in \cite{lu1} and the large data existence of global weak solutions with locally finite total variation for \eqref{e1.3} with \eqref{e1.2} was  by DiPerna \cite{di1} for some general pressure function , say $p(\rho)=k^2\rho^\gamma,  \gamma \in (1,3)$.\\
Now following \cite{Li1}, in our present work we consider the scalar function $P$ is not only a function of density $\rho$ but also a small parameter $\epsilon > 0$ satisfying
\begin{equation*}
\lim_{\epsilon\rightarrow 0} P(\rho,\epsilon)=0
\end{equation*}
and we redefine $P(\rho,\epsilon)$ as
\begin{equation}
P(\rho,\epsilon)=\epsilon p(\rho)
\label{1.6}
\end{equation}
where $p(\rho)$ is defined as before,
\begin{equation*} 
p(\rho)=\int_{0}^{\rho}{\frac{q^{'}(\xi)}{\xi}} d\xi
\end{equation*}
At this point system \eqref{e1.3} can be expressed as
\begin{equation}
\begin{aligned}
u_t +(\frac{u^2}{2}+ \epsilon p(\rho))_x &=0\\
\rho_t +(\rho u)_x&=0.
\end{aligned}
\label{e1.8}
\end{equation}
One can readily see that as $\epsilon\rightarrow 0$ formally the system \eqref{e1.8} becomes  
\begin{equation}
\begin{aligned}
u_t +(\frac{u^2}{2})_x &=0,\,\,\, x\in \mathbb{R}, t>0\\
\rho_t +(\rho u)_x&=0,\,\,\, x\in \mathbb{R}.
\end{aligned}
\label{e1.1}
\end{equation}

The above equation is a one dimensional model for the large scale structure formation of universe, see \cite{z1}. Note that the system \eqref{e1.8} can also be viewed as a  strictly hyperbolic approximation of system \eqref{e1.1}. On the other hand, one can perturb the flux function of the first equation of system \eqref{e1.1} with \eqref{1.6} to make it strictly hyperbolic.  So in the rest of our article the term "perturbed problem" means system \eqref{e1.8}.\\
From the viewpoint of hyperbolic conservation laws the  limit system \eqref{e1.1} loses strict hyperbolicity and does not have weak solution in BV- class \cite{k1, j2,  ma1, ma2}. Exact solution of the the system \eqref{e1.1} with the initial data \eqref{e1.2} was studied  by many authors \cite{j2, ma1, ma2, p1}. A different approach towards the solution of \eqref{e1.1} in the sense of Colombeau \cite{co1,ob2} can be found in \cite{o1}.
The solution for the first equation in \eqref{e1.1} is well understood   in the distributional sense \cite{h1}, whereas the solution for the second equation does not belong to the space of BV functions. In fact the second component contains $\delta$-measures. 
So one cannot expect that the product $\rho u$  can be defined in the usual sense. This is taken care of using Volpert superposition, see \cite{v1}. Moreover, non-conservative products are also discussed in \cite{le1, d1}. 

In this paper we want to determine the distributional limit of the solutions of \eqref{e1.8} when the initial data \eqref{e1.2}  are of Riemann type, i.e, 
\begin{equation}
\begin{pmatrix}
         u_0 (x)   \\
            \rho_0 (x) \\
         \end{pmatrix}
   =
\begin{cases}
\begin{pmatrix}
         u_l  \\
            \rho_l \\
         \end{pmatrix},\,\,\,\,\,\,\, \textnormal{if}\,\,\,\,\,\,\  x<0\\
 \begin{pmatrix}
         u_r   \\
            \rho_r \\
         \end{pmatrix},\,\,\,\,\,\,\, \textnormal{if}\,\,\,\,\,\,\  x>0.
\end{cases}
\label{e1.4}
\end{equation}
It turns out that this limit is a solution for \eqref{e1.1} and agrees with vanishing viscosity limit \cite{j2}. So we attempt a different approach to find measure valued solution of the system \eqref{e1.1} by passing to the limit as $\epsilon\rightarrow 0$ in the solution of an existing strictly hyperbolic model.  The theory is well developed \cite{ba1,da1} for strictly hyperbolic system  and can be used to solve \eqref{e1.8}. This kind of approach has been extensively used in the theory of isentropic gas dynamics (\cite{Li1},\cite{n1} and the references therein).\\
The paper finishes with another approximation, by adding $\epsilon>0$ in the flux function, which looks simpler than the previous one. The limit of solutions has been explored which works quite well for the rarefaction case. For the shock case the approximation of system \eqref{e1.8} can not be solved in BV class for
all types of Riemann data. Delta-waves are introduced to such cases by properly defining $u$ along the discontinuity curve. Note that this is not the usual Volpert superposition \cite{v1} and  its limit agrees with vanishing viscosity limit.\\
Our paper is organized as follows. In section 2, shock and rarefaction curves are described for system \eqref{e1.8} and dependence of the Riemann solution on $\epsilon>0$ is examined. In section 3, shock-waves are constructed  for \eqref{e1.8}-\eqref{e1.4} when $u_l > u_r$  and the limit is obtained whenever the perturbation vanishes. In section 4, entropy-entropy flux pairs are found for perturbed model \eqref{e1.8} and limit is investigated for small $\epsilon$. Section 5 contains the solutions by other elementary waves. Finally, in section 6, we discuss another approximation mentioned above.
\section{The Riemann solution}
The co-efficient matrix $A(u,\rho)$  of  the equation \eqref{e1.8}  is  given by 
\begin{equation*}
A(u,\rho)=\quad
\begin{pmatrix}
u & {\epsilon}p'(\rho) \\
\rho & u
\end{pmatrix}.
\quad  
\end{equation*}
Eigenvalues for this co-efficient matrix are the following:
$\lambda_1(u,\rho)=u-\sqrt{{\epsilon}p'(\rho)\rho}$  and   $\lambda_2(u,\rho)=u+\sqrt{{\epsilon}p'(\rho)\rho}$  and the eigenvectors to $\lambda_1$ and $\lambda_2$ are $X_1=(-\sqrt{\frac{{\epsilon}p'(\rho)}{\rho}},1)$ and $X_2=(\sqrt{\frac{{\epsilon}p'(\rho)}{\rho}},1)$ respectively and  $\nabla\lambda_i . X_i\neq 0$ for $i=1,2$.\\
Each characteristics field is genuinely nonlinear for problem \eqref{e1.8}. \\\\
\textbf{Shock curves}: The shock curves $s_1$,$s_2$ through $(u_l,\rho_l)$ are derived from the Rankine-Hugoniot conditions
\begin{equation}
\begin{aligned}
\lambda(u-u_l)=&(\frac{u^2}{2}+\epsilon p(\rho))-(\frac{u_l^2}{2}+\epsilon p(\rho_l))\\
\lambda(\rho-\rho_l)=&\rho u-\rho_l u_l.
\label{2.1}
\end{aligned}
\end{equation}
Eliminating $\lambda$ from \eqref{2.1}, shock curves are computed as
\begin{equation}
s_1=\big\{(u,\rho):(u-u_l)^2\frac{(\rho+\rho_l)}{2}=\epsilon(\rho-\rho_l)(p(\rho)-p(\rho_l)),\,\,\, \rho > \rho_l\big\}
\end{equation}
\begin{equation}
s_2=\big\{(u,\rho):(u-u_l)^2\frac{(\rho+\rho_l)}{2}=\epsilon(\rho-\rho_l)(p(\rho)-p(\rho_l)),\,\,\, \rho < \rho_l\big\}
\end{equation}\\
\textbf{Rarefaction curves}: The Rarefaction curves $R_1$, $R_2$ passing through $(u_l,\rho_l)$ are the following :\\
\textit{1- Rarefaction curve}: First Rarefaction curve passing through $(u_l,\rho_l)$ is derived by solving;
\begin{equation}
\frac{du}{d\rho}=-\sqrt{\frac{{\epsilon}p'(\rho)}{\rho}},\,\,\,\,\,\,\,\,\,\,\, u(\rho_l)=u_l.
\end{equation}
\begin{equation}
R_1=\big\{(u,\rho):u-u_l=-\int_{\rho_l}^{\rho} \sqrt{\frac{\epsilon p^{'}(\xi)}{\xi}} d\xi,\,\,\, \rho < \rho_l\big\}
\end{equation}
\textit{2- Rarefaction curve}: Second Rarefaction curve $R_2$ passing through $(u_l,\rho_l)$ is derived by solving;
\begin{equation}
\frac{du}{d\rho}=\sqrt{\frac{{\epsilon}p'(\rho)}{\rho}},\,\,\,\,\,\,\,\,\,\,\, u(\rho_l)=u_l.
\end{equation}
\begin{equation}
R_2=\big\{(u,\rho):u-u_l=\int_{\rho_l}^{\rho} \sqrt{\frac{\epsilon p^{'}(\xi)}{\xi}} d\xi,\,\,\, \rho > \rho_l\big\}
\end{equation}
To solve the equation \eqref{e1.8} with \eqref{e1.4}, three cases are required to be considered, that is (I)   $u_l>u_r$, (II) $u_l=u_r$ and   (III) $u_l<u_r$. For case (I) we have solution as a combination of two shock waves, for case (II) solutions  are given as the combination of 1-rarefaction and 2-shock curves or 1-shock and 2-rarefaction curves depending upon $\rho_l>\rho_r$ or $\rho_l<\rho_r$ respectively. And finally in case (III) solution consists of two rarefaction waves and vacuum state. In each case limit has been found and it is exactly equal to the vanishing viscosity limit found in \cite{j2} which satisfies our expectation.

\section {Formation of shock waves for $u_l >u_r$}
In this section the limiting behavior for the solution of \eqref{e1.8}-\eqref{e1.4} for $u_l> u_r$ as $\epsilon\rightarrow0$ has been studied. We first find solution for the system \eqref{e1.8} satisfying Lax- entropy condition for case $u_l >u_r$. $\rho_l$ and $\rho_r$ are taken positive through out this section.
The key result of this section is the following.

\begin{theorem}
If $u_l > u_r$, there exists a $\eta>0$ such that for any $\epsilon < \eta$, we have a unique intermediate state $(u^ *_{\epsilon}, \rho^ *_{\epsilon})$  which connects $( u_l, \rho_l)$  to  $(u^ *_{\epsilon}, \rho^ *_{\epsilon})$  by 1-shock and  $(u^ *_{\epsilon}, \rho^ *_{\epsilon})$  to $( u_r, \rho_r)$ by 2-shock and satisfies Lax-entropy condition.
\end{theorem}
\begin{proof}
The admissible 1-shock curve passing through $(\bar{u}, \bar{\rho})$ satisfies the following:

\begin{equation}
\begin{aligned}
(u-\bar{u})s_1=&(\frac{u^2}{2}+\epsilon p(\rho))-(\frac{\bar{u}^2}{2}+\epsilon p(\bar{\rho}))\\
(\rho-\bar{\rho})s_1=&\rho u-\bar{\rho}\bar{u},
 \label{e2.1}
\end{aligned}
\end{equation}
and satisfies the inequality
\begin{equation}
 s_1 < \lambda_1 (\bar{u}, \bar{\rho}), \,\,   \lambda_1 (u, \rho) < s_1 < \lambda_2 (u, \rho).
\label{ee2.2}
\end{equation}

Eliminating $s_1$  from \eqref{e2.1} and simplifying yields
\begin{equation}
(u-\bar{u})^2=2\epsilon\frac{\rho-\rho_l}{\rho+\rho_l}(p(\rho)-p(\rho_l))
\label{e2.3}
\end{equation}
We show that for a given $u< \bar{u}$, there exists a unique $\rho > \bar{\rho}$ such that equation \eqref{e2.3} holds. For that let us define a function
\begin{equation}
F(\rho):= 2\epsilon\frac{\rho-\bar{\rho}}{\rho+\bar{\rho}}(p(\rho)-p(\bar{\rho}))
\label{e2.4}
\end{equation}
We see that $F(\bar{\rho})=0$ and $F(\rho)\rightarrow \infty$ as $\rho \rightarrow\infty$. So by intemediate value theorem we have $F([\bar{\rho},\infty))=[0,\infty)$.  Hence for a given $u$ there exist a $\rho>\bar{\rho}$ such that
\begin{equation*}
F(\rho)=(u-\bar{u})^2.
\end{equation*}
This proves existence. To prove the uniqueness, now differentiate the equation \eqref{e2.4} with respect to $\rho$ to get
\begin{equation*}
F^{\prime}(\rho)= 2\epsilon \frac{2\bar{\rho}}{(\rho+\bar{\rho})^2}(p(\rho)-p(\bar{\rho}))+ 2\epsilon\frac{\rho-\bar{\rho}}{\rho+\bar{\rho}}p^{'}(\rho)
\end{equation*}
As $\rho$$>$$\bar{\rho}$ and $p^{'}(\rho)$ $>$ $0$,  $F^{\prime}(\rho)$ is positive. So $(u-\bar{u})^2$ will be achieved only once in the interval $[\bar{\rho}, \infty)$, which proves the uniqueness. The condition \eqref{e2.1} and \eqref{ee2.2} holds iff $u \leq \bar{u}$ and $\rho \geq \bar{\rho}$. In fact, 

 $s_1$ satisfies \eqref{ee2.2} if
\begin{equation}
\begin{aligned}
&  \frac{\rho u - \bar{\rho}  \bar{u}}{\rho -  \bar{\rho}} <   \bar{u} -\sqrt{\epsilon p^{' }( \bar{\rho})} \bar{\rho}\\
&  u -\sqrt{\epsilon p^{' }(\rho)}\rho<\frac{\rho u - \bar{\rho}  \bar{u}}{\rho -  \bar{\rho}}< u+\sqrt{\epsilon p^{\prime}(\rho)\rho}\\
\label{e2.6}
\end{aligned}
\end{equation}
The inequality \eqref{e2.6} holds if
\begin{equation}
\begin{aligned}
 \epsilon p^{' }(\bar{\rho})\bar{\rho} < \frac{\rho^2 (u- \bar{u})^2}{(\rho-\bar{ \rho})^2} <\epsilon p^{' }(\rho)\rho.
\end{aligned}
\label{e2.7}
\end{equation}
 Since $(u,\rho)$ satisfies \eqref{e2.4}, \eqref{e2.7} holds if 
\begin{equation*}
 p^{' }(\bar{\rho})\bar{\rho} <  \frac{2 \rho^2 (p(\rho)-p(\bar{\rho}))}{(\rho -\bar{\rho})(\rho+\bar{\rho})}<  p^{' }(\rho)\rho.
\end{equation*}
The above is true since $p$ and $p^{\prime}$ is an increasing function.

Therefore  the branch of the curve satisfying \eqref{e2.1} and \eqref{ee2.2} can be parameterized by a $C^1$ function $\rho_1: (-\infty, \bar{u}] \rightarrow [\bar{\rho}, \infty)$
with parameter $u$.

 From the equation \eqref{e2.3}, $\rho_1(u)$ satisfies 

 \begin{equation}
\frac{(u-u_l)^2}{\epsilon} \frac{\rho_l +\rho_{1}(u)}{2(\rho_{1}(u)-\rho_l))}+p(\rho_l)= p(\rho_{1}(u)).
\label{e2.10}
\end{equation}

Differentiating the above equation with respect to $u$, we get
\begin{equation*}
\frac{(u-u_l)(\rho_l +\rho_1(u))}{\epsilon(\rho_1(u)-\rho_l)}+\frac{- \rho_l \rho_1^{{\prime}}(u)(u-u_l)^2}{4 \epsilon (\rho(u)-\rho_l)}= p^{\prime}(\rho_1(u)) \rho_1^{{\prime}}(u).
\end{equation*}
Since $\rho_1 (u)>\rho_l$ ,  $\rho_l +\rho_1(u)$ and $-\rho_l +\rho_1(u)$ are positive. This implies  $\rho_1^{{\prime}}(u)$ is negative, because $p^{\prime}$ is positive.

Similarly the branch of the curve satisfying 
\begin{equation*} 
s_1 >\lambda_2 ({u},  {\rho}), \,\,   \lambda_1 (\bar{u}, \bar{\rho}) < s_1 < \lambda_2 (\bar{u}, \bar{\rho}).
\end{equation*} 
is the admissible 2-shock curve which can be parameterized by a $C^1$ function $\rho_2: (-\infty, \bar{u}] \rightarrow (-\infty, \bar{\rho}]$
with parameter $u$. 

Also $\rho_2$ satisfies the following equation: 
\begin{equation}
\frac{(u-\bar{u})^2}{\epsilon} \frac{\bar{\rho}+\rho_{2}(u)}{2(\rho_{2}(u)-\bar{\rho}))}+p(\bar{\rho})= p(\rho_{2}(u)).
\label{e2.9}
\end{equation}

 Differentiating the above equation \eqref{e2.9} with respect to $u$, we get
\begin{equation*}
\frac{(u-\bar{u})(\bar{\rho}+\rho_2(u))}{\epsilon(\rho_2(u)-\bar{\rho})}+\frac{- \bar{\rho}\rho_2{{\prime}}(u)(u-\bar{u})^2}{4 \epsilon (\rho_2(u)-\bar{\rho})}= p^{\prime}(\rho_2(u)) \rho_2{{\prime}}(u).
\end{equation*}
Since $\rho_2(u)> \bar{\rho}$,  $\bar{\rho}+\rho_2(u)$ and $-\bar{\rho}+\rho_2(u)$ are positive. This implies  $\rho_2^{{\prime}}(u)$ is positive, because $p^{\prime}$ is positive.

Consider the branch of the curve passing through $(u_r, \rho_r)$ satisfying the condition $u>u_r,\,\,\, \rho > \rho_r$. In a similar way as above it can be parameterized by a $C^1$- curve
$\rho^{*}_2(u)$.  The part of the curve $\rho^{*}_2$ from $(w, z)$  to $(u_r, \rho_r)$ will be the admissible 2-shock curve connecting $(w, z)$ to $(u_r, \rho_r)$.  So it is clear that 
${\rho^*_2}^ {\prime}(u)$ is positive.

Let's denote admissible 1-shock curve passing through $(u_l, \rho_l)$ as $\rho^{*}_1 $.  As from the previous analysis this is parameterized by  $\rho^{*}_1 :  (- \infty, u_l] \rightarrow [\rho_l, \infty)$ and satisfies ${\rho^*_1}^ {\prime}(u)< 0$ . 

$\rho^{*}_1(u_r)$  satisfies \eqref{e2.9} with $\rho_1(u)$  and $\bar{u}$replaced by $\rho^{*}_1 (u_r)$ and $u_r$ respectiveily,  and $\rho^{*}_2 (u_l)$  satisfies  \eqref{e2.10} with 
$\rho_2(u)$ and $\bar{u}$ replaced by $\rho^{*}_2 (u_l)$ and $u_l$ respectiveily. Hence $\rho^{*}_1(u_r)$ and  $\rho^{*}_2 (u_l)$  goes to $\infty$ as $\epsilon$ tends to zero. Therefore there exists a $\eta >0$  such that $\epsilon< \eta, $  we have  $ \rho^{*}_2 (u_l)>\rho_l$ and  $ \rho^{*}_1 (u_r)>\rho_r$. Now consider the function $\rho^{*}_1- \rho^{*}_2$. Since $ \rho^{*}_1(u_l)-\rho^{*}_2(u_l)=\rho^{*}_1 -\rho^{*}_2 (u_l) <0$  and  $ \rho^{*}_1(u_r)-\rho^{*}_2 (u_r)= \rho^{*}_1(u_r)-\rho_r >0$.  Now by intermediate value theorem there exist a point $ u^{*}_{\epsilon}$ such that  $\rho^{*}_1( u^{*}_{\epsilon})=\rho^{*}_2 ( u^{*}_{\epsilon})=\rho^{*}_{\epsilon}$(say). 
 $\rho^{*}_{\epsilon}$ is unique because $\rho^{*}_1$  is stricty decreasing and $\rho^{*}_2$ is strictly increasing. Since we are cosidering only admissible curves  lax entropy condition 
holds. This completes the proof.\\
\end{proof}
 
Now we determine the limit of the problem \eqref{e1.8} for the shock case. For this  first we will define $\delta$-distribution followed by a simple technical lemma which will be useful later.

\textbf{Definition}: A weighted $\delta$-distribution "$ d(t)\delta_{x=c(t)}$" is concentrated on a smooth curve $x=c(t)$ can be defined by
\begin{equation*}
\langle\, d(t)\delta_{x=c(t)},\varphi(x,t)\rangle=\int_{0}^{\infty}d(t)\varphi(c(t),t)dt
\end{equation*}
for all $\varphi\in C^\infty_c(\mathbb{R}\times (0,\infty))$
\begin{lem}
Suppose $a_{\epsilon}(t)$ and $b_{\epsilon}(t)$ converges uniformly to $0$ on compact subsets of $(\mathbb{R}\times (0,\infty))$ as $\epsilon$ tends to zero.  Also assume that $d_{\epsilon}(t)$ conveges to $d(t)$  uniformly on compact subsets of $(\mathbb{R}\times (0,\infty))$ as $\epsilon$ tends to zero, then 
\begin{equation*}
\frac{1}{b_{\epsilon}(t)-a_{\epsilon}(t)} d_{\epsilon}(t)
\displaystyle{\chi_{(c(t)-a_{\epsilon}(t), c(t)+b_{\epsilon}(t))}}(x)
\end{equation*}
converges to $d(t)\delta_{x=c(t)}$ in the sense of distribution.
\end{lem}
\begin{proof}
Denote 
\begin{equation*}
\Psi(x,t)=\frac{1}{b_{\epsilon}(t)-a_{\epsilon}(t)}d_{\epsilon}(t)
\chi_{(c-a_{\epsilon}(t), c+b_{\epsilon}(t))}(x).
\end{equation*}
Now consider the integral
\begin{equation*}
\begin{aligned}
&\Big|\int_{0}^{\infty}\int_{-\infty}^{\infty} \big(\Psi(x,t)\varphi(x,t) dx dt-\int_{0}^ {\infty} d(t)\varphi(c(t),t) dt \big)\Big|\\
&\leq\int_{0}^{\infty}\Big|\frac{1}{b_{\epsilon}(t)-a_{\epsilon}(t)}\int_{c(t)-a_{\epsilon}(t)}^{c(t)+b_{\epsilon}(t)}d_{\epsilon}(t)\varphi(x,t)-d(t)\varphi(c(t),t)\Big|\,\,dx\,dt
\end{aligned}
\label{e2.12}
\end{equation*}
Now since $\varphi(x,t)$ has compact support and $d_{\epsilon}(t)$ converges to $d(t)$  uniformly on compact sets as $\epsilon\rightarrow0$,  the last integral converges to $0$.
Since this is true for all test function $\varphi$, the proof of this lemma is completed .
\end{proof}

\begin{theorem}
The distribution limit $(u^{\epsilon}, \rho^{\epsilon})$ exists  as $\epsilon$ approaches zero and is given by $(u, \rho)$.
\begin{equation}
u(x,t)=\begin{cases}
u_l,\,\,\, \textnormal{if} x< \frac{u_l +u_r}{2}t\\
\frac{u_l+u_r}{2},\,\,\,  \textnormal {if}\,\,\,   x=\frac{u_l+u_r}{2}t\\
u_r,\,\,\, \textnormal{if} x> \frac{u_l +u_r}{2}t
\label{eq3.7}
\end{cases}
\end{equation}
and 
\begin{equation}
\rho(x,t)=\begin{cases}
\rho_l,\,\,\, \textnormal{if}\,\,\,\, x< \frac{u_l +u_r}{2}t\\
(u_l-u_r)\frac{\rho_l +\rho_r}{2} \delta_{x=\frac{u_l +u_r}{2}t},\,\,\, \textnormal{if}\,\,\,\, x= \frac{u_l +u_r}{2}t\\
\rho_r,\,\,\, \textnormal{if}\,\,\,\, x> \frac{u_l +u_r}{2}t
\end{cases}
\label{eq3.8}
\end{equation}
\end{theorem}
\begin{proof}
From the above theorem $(u^{*}_{\epsilon}, \rho^{*}_{\epsilon})$ satisfies the following conditions.
\begin{equation}
\begin{aligned}
(u^{*}_{\epsilon}-u_l)\frac{\rho^{*}_{\epsilon}  u^{*}_{\epsilon} -\rho_l u_l}{\rho^{*}_{\epsilon} -\rho_l}&=(\frac{{u^{*}_{\epsilon}}^2}{2}+\epsilon p(\rho^{*}_{\epsilon}))-(\frac{u_l^2}{2}+\epsilon p(\rho_l))\\
(u^{*}_{\epsilon}-u_r)\frac{\rho^{*} u^{*}_{\epsilon}-\rho_r u_r}{\rho^{*}_{\epsilon}-\rho_r}&=(\frac{{u^{*}_{\epsilon}}^2}{2}+\epsilon p(\rho^{*}_{\epsilon}))-(\frac{u_r^2}{2}+\epsilon p(\rho_r)).
\end{aligned}
\label{e2.11}
\end{equation}

We know $u^{*}_{\epsilon} \in (u_r, u_l)$.  So the sequence $u^{*}_{\epsilon}$ is bounded. Now our claim is that  $\rho^{*}_{\epsilon}$ is unbounded as $\epsilon$ tends to zero.\\
\textit{proof of the claim(2.1)}: Suppose $\rho^*_{\epsilon}$ is bounded.Then it has a convergent subsequence still denoted by $\rho^*_{\epsilon}$ and it converges to $\rho^*$ as $\epsilon \rightarrow 0$. Then from the equation (4.1) we get that $\rho^*_{\epsilon}$ satisfies:
\begin{equation}
(u^*_{\epsilon}-u_l)^2\frac{(\rho^*_{\epsilon}+\rho_l)}{2}=\epsilon(\rho^*_{\epsilon}-\rho_l)(p(\rho^*_{\epsilon})-p(\rho_l))
\end{equation}
Now as $\epsilon \rightarrow 0$,the above equation becomes
\begin{equation}
(u^*-u_l)^2\frac{(\rho^*+\rho_l)}{2}=0,
\end{equation}
as R.H.S of the equation is bounded. Now since $\rho^*+\rho_l >0$, we get $u^*= u_l$.
Again,since $\rho^*_{\epsilon}$ satisfies:
\begin{equation}
(u^*_{\epsilon}-u_r)^2\frac{(\rho^*_{\epsilon}+\rho_r)}{2}=\epsilon(\rho^*_{\epsilon}-\rho_r)(p(\rho^*_{\epsilon})-p(\rho_r))
\end{equation}
By similar argument we get, $u^*=u_r$. So $u_l=u_r$. This leads to a contradiction.\\\\
 So for a subsequence let  $u^{*}_{\epsilon}$ converges to $u^{*}$  and  $\rho^{*}_{\epsilon}$ tend to $+\infty$. Passing to the limit  for this subsequence in \eqref{e2.11}, we get
\begin{equation*}
\begin{aligned}
 u^{*}(u^{*}-u_l)&= \frac{{u^{*}}^2}{2}-\frac{{u_{l}}^2}{2}+l\\
 u^{*}(u^{*}-u_r)&= \frac{{u^{*}}^2}{2}-\frac{{u_{r}}^2}{2}+l,
\end{aligned}
\end{equation*}
where $\displaystyle {\lim_{\epsilon \rightarrow 0}} \epsilon p(\rho^{*}_{\epsilon})=l.$
Solving the above two equations we get
\begin{equation}
 u^{*}=\frac{u_l +u_r}{2}\,\, \textnormal{and}\,\, l=\frac{1}{8}(u_l -u_r)^2.
\label{e2.15}
\end{equation}

The solution for $(u^{\epsilon},\rho^{\epsilon })$ is given by
\begin{equation}
 \begin{aligned}
u^{\epsilon}(x,t)=\begin{cases} 
                            u_l  \,\,\,\,\textnormal{ if} \,\,\,\, x< ( \frac{{u^*}_{\epsilon} + u_l}{2} + \frac{\epsilon(p(\rho^{*}_{\epsilon})-p(\rho_l))}{u^{*}_{\epsilon}-u_l})t\\
                            u_{\epsilon}^* \,\,\,\, \textnormal{ if}\,\,\,\, ( \frac{{u^*}_{\epsilon} + u_l}{2} + \frac{\epsilon(p(\rho^{*}_{\epsilon})-p(\rho_l))}{u^{*}_{\epsilon}-u_l})t<x <
(\frac{{u^*}_{\epsilon} + u_r}{2} + \frac{\epsilon(p(\rho^{\star}_{\epsilon})-p(\rho_r))}{u^{*}_{\epsilon}-u_r})t\\
                            u_r \,\,\,\, \textnormal{ if} \,\,\,\, x> (\frac{{u^*}_{\epsilon} + u_r}{2} + \frac{\epsilon(p(\rho^{*}_{\epsilon})-p(\rho_r))}{u^{*}_{\epsilon}-u_r})t.
                           \end{cases}
 \end{aligned}
\label{e3.14}
\end{equation}
 \begin{equation}
 \begin{aligned}
\rho^{\epsilon}(x,t)=\begin{cases} 
                            \rho_l  \,\,\,\,\textnormal{ if} \,\,\,\, x< ( \frac{{u^*}_{\epsilon} + u_l}{2} + \frac{\epsilon(p(\rho^{*}_{\epsilon})-p(\rho_l))}{u^{*}_{\epsilon}-u_l})t\\
                            \rho_{\epsilon}^* \,\,\,\, \textnormal{ if}\,\,\,\, ( \frac{{u^*}_{\epsilon} + u_l}{2} + \frac{\epsilon(p(\rho^{*}_{\epsilon})-p(\rho_l))}{u^{*}_{\epsilon}-u_l})t<x <
(\frac{{u^*}_{\epsilon} + u_r}{2} + \frac{\epsilon(p(\rho^{\star}_{\epsilon})-p(\rho_r))}{u^{*}_{\epsilon}-u_r})t\\
                            \rho_r \,\,\,\, \textnormal{ if} \,\,\,\, x> (\frac{{u^*}_{\epsilon} + u_r}{2} + \frac{\epsilon(p(\rho^{*}_{\epsilon})-p(\rho_r))}{u^{*}_{\epsilon}-u_r})t.
                           \end{cases}
 \end{aligned}
\label{e3.15}
\end{equation}
As $u^*_{\epsilon}$ converges to $u^*=\frac{u_l+u_r}{2}$ as $\epsilon \rightarrow 0$, we have the limit for $u(x,t)$ as stated in the theorem.\\
From \eqref{e2.11} and \eqref{e2.15}, one can show that
 $$\displaystyle {\lim_{\epsilon \rightarrow 0}}[ \frac{{u^*}_{\epsilon} + u_l}{2} + \frac{\epsilon(p(\rho^{*}_{\epsilon})-p(\rho_l))}{u^{*}_{\epsilon}-u_l}]=\frac{u_l +u_r}{2},$$
and $$\displaystyle {\lim_{\epsilon \rightarrow 0}} \frac{{u^*}_{\epsilon} + u_r}{2} + \frac{\epsilon(p(\rho^{*}_{\epsilon})-p(\rho_r))}{u^{*}_{\epsilon}-u_r}
=\frac{u_l +u_r}{2}.$$

Let's denote,

\begin{equation}
\begin{aligned}
c(t)&= \frac{u_l + u_r}{2}t\\
a_{\epsilon}(t)&= c(t)-  ( \frac{{u^*}_{\epsilon} + u_l}{2} + \frac{\epsilon(p(\rho^{*}_{\epsilon})-p(\rho_l))}{u^{*}_{\epsilon}-u_l})t\\
b_{\epsilon}(t)&= (\frac{{u^*}_{\epsilon} + u_r}{2} + \frac{\epsilon(p(\rho^{*}_{\epsilon})-p(\rho_r))}{u^{*}_{\epsilon}-u_r})t- c(t)\\
d_{\epsilon}(t)&=\Big[\frac{u_r- u_l}{2}+ \frac{\epsilon(p(\rho^{*}_{\epsilon})-p(\rho_r))}{u^{*}_{\epsilon}-u_r}-
 \frac{\epsilon(p(\rho^{*}_{\epsilon})-p(\rho_l))}{u^{*}_{\epsilon}-u_l}\Big] \rho_{\epsilon}^*t.\\
\end{aligned}
\end{equation}

With the above notations the formula for $\rho^{\epsilon}$ in equation \eqref{e3.15} can be written in the following form as in the Lemma.
\begin{equation}
\begin{aligned}
\rho^{\epsilon}=& \rho_l \displaystyle{\chi_{(-\infty, c(t)- a_{\epsilon}(t))}}(x)+\frac{ d_{\epsilon}(t)}{b_{\epsilon}(t)- a_{\epsilon}(t)}\displaystyle{\chi_{(c(t)-a_{\epsilon}(t), c(t)+b_{\epsilon}(t))}}(x)\\
&+\rho_r\displaystyle{\chi_{(c(t)+b_{\epsilon}(t), \infty)}}(x).
\end{aligned}
\label{e3.17}
\end{equation}

 Now we will determine the limit of $d_{\epsilon}(t)$ as $\epsilon$ tends to zero.

The equation \eqref{e2.11} can also be written in the following  form.

 \begin{equation}
\begin{aligned}
&(\rho_{\epsilon}^{*}-\rho_l)\Big\{ \frac{u_{\epsilon}^{*}+u_l}{2}+ \frac{\epsilon(p(\rho_{\epsilon}^{*})-(\rho_l))}{u_{\epsilon}^{*}-u_l}\Big\}
=\rho_{\epsilon}^{*}u_{\epsilon}^{*}- \rho_l u_l\\
&(\rho_{\epsilon}^{*}-\rho_r)\Big\{ \frac{u_{\epsilon}^{*}+u_r}{2}+ \frac{\epsilon(p(\rho_{\epsilon}^{*})-(\rho_r))}{u_{\epsilon}^{*}-u_r}\Big\}
=\rho_{\epsilon}^{*}u_{\epsilon}^{*}- \rho_r u_r.
\label{e2.19}
\end{aligned}
\end{equation}
Subtracting second equation from the first in \eqref{e2.19}, we get
\begin{equation}
\begin{aligned}
&\Big[\frac{u_r- u_l}{2}+ \frac{\epsilon(p(\rho^{*}_{\epsilon})-p(\rho_r))}{u^{*}_{\epsilon}-u_r}-
 \frac{\epsilon(p(\rho^{*}_{\epsilon})-p(\rho_l))}{u^{*}_{\epsilon}-u_l}\Big] \rho_{\epsilon}^*\\
&= \rho_l u_l -  \rho_r u_r +\rho_r ( \frac{u_{\epsilon}^{*}+u_r}{2})-\rho_l( \frac{u_{\epsilon}^{*}+u_l}{2})\\
&+\rho_r \frac{\epsilon(p(\rho_{\epsilon}^{*})-p(\rho_r))}{u_{\epsilon}^{*}-u_r}-\rho_l \frac{\epsilon(p(\rho_{\epsilon}^{*})-p(\rho_l))}{u_{\epsilon}^{*}-u_l}.
\end{aligned}
\end{equation}
Passing to the limit as $\epsilon \rightarrow  0$,  we get
\begin{equation}
 \displaystyle {\lim_{\epsilon \rightarrow 0}} \Big[\frac{u_r- u_l}{2}+ \frac{\epsilon(p(\rho^{*}_{\epsilon})-p(\rho_r))}{u^{*}_{\epsilon}-u_r}-
 \frac{\epsilon(p(\rho^{*}_{\epsilon})-p(\rho_l))}{u^{*}_{\epsilon}-u_l}\Big] \rho_{\epsilon}^*=\frac{1}{2}(u_l -u_r)(\rho_l + \rho_r)
\label{e2.21}
\end{equation}
This implies 
\begin{equation}
\displaystyle {\lim_{\epsilon \rightarrow 0}} b_{\epsilon}(t)= \frac{1}{2}(u_l -u_r)(\rho_l + \rho_r)t.
\label{e2.22}
\end{equation}
Here in the calculation \eqref{e2.22}, we have used the fact that  $\displaystyle {\lim_{\epsilon \rightarrow 0}}  \epsilon p(\rho^{*}_{\epsilon})=\frac{1}{8}(u_l-u_r)^2$ and $ \displaystyle {\lim_{\epsilon \rightarrow 0}} u^{*}_{\epsilon} =\frac{u_l +u_r}{2}$ from the equation \eqref{e2.15}.

The first and third term of \eqref{e3.17} converges to  $\rho_l \displaystyle{\chi_{(-\infty, \frac{u_l + u_r}{2}t)}}(x)$  and \\
 $ \rho_r \displaystyle{\chi_{( \frac{u_l + u_r}{2}t, \infty)}}(x)$ respectively.
Hence employing the above lemma to the middle term of \eqref{e3.17}, we get the distribution limit $\rho(x,t)$ as given in the theorem. Note that all the analysis  has been done for a subsequence, but since limit is same for any subsequence, this implies full sequence converges. The proof of theorem 3.3 is completed.

\end{proof}

\section{entropy and entropy flux pairs}
In this section we explicitly find the entropy and entropy flux pairs for the perturbed system \eqref{e1.8}. Let us start with a definition of entropy-entropy flux pairs.\\
\textbf{Definition}: A continuously differentiable function $\eta:\mathbb{R}^2\mapsto \mathbb{R}$ is called an \textit{entropy} for the system with \textit{entropy flux} $q:\mathbb{R}^2\mapsto \mathbb{R}$,if 
\begin{equation*}
D\eta(u) . Df(u)= Dq(u),\,\,\,\, u\in\mathbb{R}^n
\end{equation*}
where $f$ is the flux function for the system.\\
\textbf{Entropy inequality}: A weak solution $u$ of a system is called \textit{entropy admissible} if 
\begin{equation*}
\iint_{\Omega} {\eta(u)\varphi_x + q(u)\varphi_x} \,dx\,dt \geq 0
\end{equation*}
for every \textit{positive, $C^\infty$-functions} $\varphi: \Omega\rightarrow \mathbb{R}^2$ with compact support.
The above  can be restated in the following way:
\begin{equation}
\eta(u)_t+q(u)_x \leq 0
\label{4.1}
\end{equation} 
in the distributional sense for every pair $(\eta ,q)$ defined above.\\\\
 A pair $(\eta,q)$ of real valued maps is an entropy-entropy flux pair of \eqref{e1.3} if 
\begin{equation}
\begin{aligned}
\bigg(\frac{\partial \eta}{\partial u}u+\frac{\partial \eta}{\partial \rho}\rho\, ,           \,\, \epsilon\rho e^\rho \frac{\partial \eta}{\partial u}+u\frac{\partial \eta}{\partial \rho}\bigg)
=\bigg(\frac{\partial q}{\partial u},\frac{\partial q}{\partial \rho}\bigg)
\end{aligned}
\end{equation}
That is;
\begin{equation}
\begin{aligned}
\frac{\partial q}{\partial u}=\frac{\partial \eta}{\partial u}u+\frac{\partial \eta}{\partial \rho}\rho\\
\frac{\partial q}{\partial \rho}=\epsilon\rho e^\rho\frac{\partial \eta}{\partial u}+u\frac{\partial \eta}{\partial \rho}
\label{e3.2}
\end{aligned}
\end{equation}
Eliminating $q$ from \eqref{e3.2}, we have
\begin{equation}
\frac{\partial^2 \eta}{\partial \rho^2}=\epsilon e^ \rho\frac{\partial^2 \eta}{\partial u^2}.
\end{equation}

One can see that
\begin{equation}
\eta(u,\rho)=\frac{1}{2}u^2 + \epsilon e^\rho
\end{equation}
 is a solution of above equation which is a strictly convex entropy (since $D^2\eta > 0$) of the system \eqref{e1.8} and the corresponding entropy flux is 
 \begin{equation}
 q(u,\rho)=\frac{1}{3}u^3 + \epsilon \rho u e^ \rho.
 \end{equation}
 So, we need to calculate the following.
 \begin{equation}
 \begin{aligned}
 \eta_t+ q_x= -s_{1}\bigg(\frac{1}{2}u^{*2}_{\epsilon}+\epsilon e^{\rho^*_{\epsilon}}-\frac{1}{2}u^{2}_{l}-\epsilon e^{\rho_l}\bigg)\delta_{x=s_{1}t}\\
 -s_{2}\bigg(\frac{1}{2}u^{2}_{r}+\epsilon e^{\rho_r}-\frac{1}{2}u^{*2}_{\epsilon}-\epsilon e^{\rho^*_{\epsilon}}\bigg)\delta_{x=s_{2}t} \\ 
 +\bigg(\frac{1}{3}u^{*3}_{\epsilon}+\epsilon\rho^*_{\epsilon}u^*_{\epsilon}e^{\rho^*_{\epsilon}}-\frac{1}{3}u^3_l-\epsilon\rho_lu_le^{\rho_l}\bigg)\delta_{x=s_{1}t}\\
 +\bigg(\frac{1}{3}u^3_r-\epsilon\rho_ru_re^{\rho_r}-\frac{1}{3}u^{*3}_{\epsilon}-\epsilon\rho^*_{\epsilon}u^*_{\epsilon}e^{\rho^*_{\epsilon}}\bigg)\delta_{x=s_{2}t},
\label{e3.6}
 \end{aligned}
 \end{equation}
 where $s_{1}t$ and $s_{2}t$ are defined as,
 \begin{equation}
 \begin{aligned}
 s_{1}t=( \frac{{u^*}_{\epsilon} + u_l}{2} + \frac{\epsilon(p(\rho^{*}_{\epsilon})-p(\rho_l))}{u^{*}_{\epsilon}-u_l})t,\\
 s_{2}t=(\frac{{u^*}_{\epsilon} + u_r}{2} + \frac{\epsilon(p(\rho^{*}_{\epsilon})-p(\rho_r))}{u^{*}_{\epsilon}-u_r})t. 
 \end{aligned}
 \end{equation}
 Let us consider,
 \begin{equation}
 \begin{aligned}
 -s_{1}\bigg(\frac{1}{2}u^{*2}_{\epsilon}+\epsilon e^{\rho^*_{\epsilon}}\bigg)+\bigg(\frac{1}{3}u^{*3}_{\epsilon}+\epsilon\rho^*_{\epsilon}u^*_{\epsilon}e^{\rho^*_{\epsilon}}\bigg)\\
 =u^{*2}_{\epsilon}\bigg(\frac{1}{3}u^*_{\epsilon}-\frac{1}{2}s_{1}\bigg)+ \epsilon e^{\rho^*_{\epsilon}}\bigg(\rho^*_{\epsilon}u^*_{\epsilon}-s_{1}\bigg)
\label{e3.7}
 \end{aligned}
 \end{equation}
 and
 \begin{equation}
 \begin{aligned}
 s_{2}\bigg(\frac{1}{2}u^{*2}_{\epsilon}+\epsilon e^{\rho^*_{\epsilon}}\bigg)-\bigg(\frac{1}{3}u^{*3}_{\epsilon}+\epsilon\rho^*_{\epsilon}u^*_{\epsilon}e^{\rho^*_{\epsilon}}\bigg)\\
 =u^{*2}_{\epsilon}\bigg(-\frac{1}{3}u^*_{\epsilon}+\frac{1}{2}s_{2}\bigg)+\epsilon e^{\rho^*_{\epsilon}}\bigg(-\rho^*_{\epsilon}u^*_{\epsilon}+s_{2}\bigg).
\label{e3.8}
 \end{aligned}
\end{equation}
Observe that as $\epsilon\rightarrow0$, the first term of the both equation \eqref{e3.7} and \eqref{e3.8} are going to the same quantity ( say $\frac{(u_l+u_r)^3}{48}$ ) with a negative sign. So it cancels each other after summing up.\\
Now the crucial part is to handle the second term of both equations. After adding   we have
\begin{equation}
\epsilon e^{\rho^*_{\epsilon}}\bigg(s_{2}-s_{1} \bigg).
\label{e3.9}
\end{equation}
We claim that as $\epsilon\rightarrow 0$, \eqref{e3.9} goes to $0$.\\
\textit{Proof of the claim}: \eqref{e3.9} can be written as 
\begin{equation}
\begin{aligned}
\epsilon e^{\rho^*_{\epsilon}}\bigg(s_{2}-s_{1} \bigg)
={\epsilon} \frac{p^{'}(\rho^*_{\epsilon})}{p(\rho^*_{\epsilon})}\frac{p(\rho^*_{\epsilon})}{\rho^*_{\epsilon}}\bigg(s_{2}-s_{1} \bigg).
\end{aligned}
\end{equation}
Now since 
\begin{equation}
 \lim_{\epsilon\rightarrow0}\frac{p^{'}(\rho^*_{\epsilon})}{p(\rho^*_{\epsilon})}=1\,\,\, \textnormal{and}\,\,\,  \lim_{\epsilon\rightarrow0}\epsilon p(\rho^*_{\epsilon})=l\,\,\,\textnormal{and}\,\,\, \rho^*_{\epsilon}\rightarrow \infty\,\,\, \textnormal{as}\,\,\, \epsilon\rightarrow0.
 \end{equation}
Moreover,
\begin{equation}
\lim_{\epsilon\rightarrow0}s_{1}=\lim_{\epsilon\rightarrow0}s_{2}.
\end{equation}
 This proves the claim.\\\\
Next consider the remaining terms of \eqref{e3.6}.
\begin{equation}
s_{1}\bigg(\frac{1}{2}u^2_l+\epsilon e^{\rho_l}\bigg)-\frac{1}{3}u^3_l-\epsilon\rho_lu_le^{\rho_l} \rightarrow \,\,\ \frac{u_l+u_r}{4}u^2_l-\frac{1}{3}u^3_l
\label{e3.11}
\end{equation}
and 
\begin{equation}
-s_{2}\bigg(\frac{1}{2}u^2_r+\epsilon e^{\rho_r}\bigg)+\frac{1}{3}u^3_l+\epsilon\rho_lu_le^{\rho_l} \rightarrow \,\,\ -\frac{u_l+u_r}{4}u^2_r-\frac{1}{3}u^3_r.
\label{e3.12}
\end{equation}
as $\epsilon\rightarrow0$.
So finally from \eqref{e3.6}, \eqref{e3.11}and \eqref{e3.12} we get,
\begin{equation*}
 \eta_t+ q_x \rightarrow \bigg(\frac{u_l+u_r}{4}(u^2_l-u^2_r)+\frac{1}{3}(u^3_r-u^3_l)\bigg)\delta _{x=\frac{u_l+u_r}{2}}
\end{equation*}
and 
\begin{equation*}
(\frac{u_l+u_r}{4}(u^2_l-u^2_r)+\frac{1}{3}(u^3_r-u^3_l)= -\frac{(u_l-u_r)^3}{12} < 0,
\end{equation*}
since $u_l>u_r$.  So for $\epsilon$ small \eqref{e3.6} satisfies \eqref{4.1}. This completes the proof.

\section{Solution for the case  $u_l \leq u_r$}

This section is devoted to discuss other two cases, i.e, $u_l = u_r$ and  $u_l < u_r$.  In this section our proof goes in the spirit of \cite{n1}.\\
\textit{CaseI $(u_l=u_r)$}:
For $u_l = u_r$, initial data is
\begin{equation}
\begin{pmatrix}
         u_0 (x)   \\
            \rho_0 (x) \\
         \end{pmatrix}
   =
\begin{cases}
\begin{pmatrix}
         u_l  \\
            \rho_l \\
         \end{pmatrix},\,\,\,\,\,\,\, \textnormal{if}\,\,\,\,\,\,\  x<0\\
 \begin{pmatrix}
         u_l  \\
            \rho_r \\
         \end{pmatrix},\,\,\,\,\,\,\, \textnormal{if}\,\,\,\,\,\,\  x>0.
\end{cases}
\end{equation}
Now if $\rho_l=\rho_r,$  we have the trivial solution $u(x,t)=u_l$ and $\rho(x,t) = \rho_l$.  Another two possibilities are  $\rho_r < \rho_l$ or  $\rho_r > \rho_l$.\\
\textit{Subcase I($\rho_r < \rho_l$)}: For this case we start traveling from the state $(u_l,\rho_l)$ and by $R_1$ we reach at $( u^*_{\epsilon}, \rho^*_{\epsilon})$, then from $( u^*_{\epsilon}, \rho^*_{\epsilon})$ we travel by $S_2$ and reach at $(u_l,\rho_r)$.
 1-rarefaction curve through $(u_l,\rho_l)$ is obtained solving the differential equation
\begin{equation}
\frac{du}{d\rho}=-\sqrt{\frac{\epsilon p^{'}(\rho)}{\rho}}\,\,\,\,\,\,\,\,\,\,\,\, u(\rho_l)=u_l
\label{5.6}
\end{equation}
So the branch of the curve satisfying \eqref{5.6} can be parameterized by a $C^{1}$ function $u_1:[\rho_r,\rho_l]\rightarrow[u_l,\infty)$ with parameter $\rho$.
Since $p^{'}(\rho) > 0$,  we see that $u_1$ is decreasing. Therefore, $u_1(\rho_r) > u_l$.\\ 
Any state $(u,\rho)$ connected to the end state  $(u_l,\rho_r)$ by admissible 2-shock curve $S_2$ satisfies the following equations;
\begin{equation}
(u-u_l)\frac{\rho u-\rho_r u_l}{\rho-\rho_r}=(\frac{u^2}{2}+\epsilon p(\rho))-(\frac{u_l^2}{2}+\epsilon p(\rho_r))
\label{5.2},
\end{equation}
\begin{equation}
s >\lambda_2 ({u},  {\rho}), \,\,   \lambda_1 (u_l,\rho_r) < s < \lambda_2 (u_l, \rho_r)
\label{5.3}
\end{equation}
\eqref{5.2} -\eqref{5.3} valid iff $u > u_l$ and $\rho > \rho_r$. Again \eqref{5.2} implies
\begin{equation}
(u-u_l)^2=2\epsilon\frac{(\rho-\rho_r)}{(\rho+\rho_r)}(p(\rho)-p(\rho_r)).
\end{equation}
Our claim is that for every fixed $\rho >\rho_r$ there exists a unique $u>u_l$ such that the equation (4.2) holds.
Let us define
\begin{equation*}
F(u):= (u-u_l)^2.
\end{equation*}
Since $F(u_l)=0$, we have $F(u) \rightarrow \infty$ as $u \rightarrow \infty$ and $F([u_l, \infty))= [0, \infty)$. Since $p$ is increasing and $\rho >\rho_r$, right hand side of (4.3) is positive. So there exists a $u>u_l$ such that 
\begin{equation*}
F(u)=2\epsilon\frac{(\rho-\rho_r)}{(\rho+\rho_r)}(p(\rho)-p(\rho_r))
\end{equation*}
Also,
\begin{equation*}
F^{'}(u)=2(u-u_l)>0,
\end{equation*}
since $u>u_l$. Therefore $u$ is unique.\\
Similarly in Theorem 3.1, the branch of the curve satisfying \eqref{5.2} and \eqref{5.3} can be parameterized by a $C^1$-function $u_2(\rho)=u_2 :[\rho_r,\rho_l]\rightarrow[u_l,\infty)$ satisfying
\begin{equation}
F(u_2(\rho))=(u_2(\rho)-u_l)^2=2\epsilon\frac{(\rho-\rho_r)}{(\rho+\rho_r)}(p(\rho)-p(\rho_r))
\label{e5.4}
\end{equation}
Note that $u_2(\rho_r)=u_l$ 
and from our argument it is clear that the function $u_2$ is well defined and our claim is that the function $u_2$ is increasing in the interval $(\rho_r,\rho_l)$
Now differentiating the above equation \eqref{e5.4} we get,
\begin{equation*}
(u_2(\rho)-u_l){u_2}^{\prime}(\rho)=\epsilon\big[p^{\prime}(\rho)\frac{\rho-\rho_r}{\rho+\rho_r}+(p(\rho)-p(\rho_r))\frac{2\rho_r}{(\rho+\rho_r)^2}\big]
\end{equation*}
Since $\rho>\rho_r$ and $p(\rho)$ is an increasing function, i.e, $p^{\prime}(\rho)>0$, RHS of above equation is $>0$ for small  $\epsilon>0$. That is,
$(u_2(\rho)-u_l){u_2}^{\prime}(\rho)>0$.
Previously we proved that for $\rho>\rho_r$ there exits a unique $u>u_l$ satisfying \eqref{e5.4}.  This implies  ${u_2}^{\prime}(\rho)>0$. This proves our claim.

From the above analysis,  there exists an intermediate state $\rho^*_{\epsilon} \in (\rho_r, \rho_l)$ such that $u_1(\rho^*_{\epsilon})=u_2(\rho^*_{\epsilon})=u^*_{\epsilon}.$
Hence the solution for \eqref{e1.8} is given by:
\begin{equation}
u^{\epsilon}=\left\{
	\begin{array}{llll}
		u_l  & \mbox{if } x < \lambda_1(u_l,\rho_l)t \\
		R^{u}_{1}(x/t)(u_l,\rho_l) & \mbox{if } \lambda_1(u_l,\rho_l)t< x < \lambda_1(u^*_{\epsilon},\rho^*_{\epsilon})t \\
                    u^*_{\epsilon} &\mbox{if } \lambda_1(u^*_{\epsilon},\rho^*_{\epsilon})t < x < \frac{\rho_r u_l-\rho^*_{\epsilon}u^*_{\epsilon}}{\rho_r-\rho^*_{\epsilon}}t \\
                     u_r                   &\mbox{if }  x>\frac{\rho_r u_l-\rho^*_{\epsilon}u^*_{\epsilon}}{\rho_r-\rho^*_{\epsilon}}t
	\end{array}
\right.
\end{equation}
and 
\begin{equation}
\rho^{\epsilon}=\left\{
	\begin{array}{llll}
		\rho_l  & \mbox{if } x < \lambda_1(u_l,\rho_l)t \\
		R^{\rho}_{1}(x/t)(u_l,\rho_l) & \mbox{if } \lambda_1(u_l,\rho_l)t< x < \lambda_1(u^*_{\epsilon},\rho^*_{\epsilon})t \\
                    \rho^*_{\epsilon} &\mbox{if } \lambda_1(u^*_{\epsilon},\rho^*_{\epsilon})t < x < \frac{\rho_r u_l-\rho^*_{\epsilon}u^*_{\epsilon}}{\rho_r-\rho^*_{\epsilon}}t \\
                     \rho_r                   &\mbox{if }  x>\frac{\rho_r u_l-\rho^*_{\epsilon}u^*_{\epsilon}}{\rho_r-\rho^*_{\epsilon}}t
	\end{array}
\right.
\end{equation}
Where $R_1(\xi)(\bar{u},\bar{\rho})=(R^{u}_{1}(\xi)(\bar{u},\bar{\rho}),R^{\rho}_{1}(\xi)(\bar{u},\bar{\rho}))$ and $R^{u}_{1}(\xi)(\bar{u},\bar{\rho})$ is obtained by solving 
\begin{equation}
\frac{du}{d\xi}=-\sqrt{\frac{{\epsilon}p'(\rho(\xi))}{\rho(\xi)}},\,\,\,\,\, u(0)=\bar{u}
\end{equation}
and
$R^{\rho}_{1}(\xi)(\bar{u},\bar{\rho})$ is obtained by solving 
\begin{equation}
\frac{d\rho}{d\xi}=1, \,\,\,\,\,\,\, \rho(0)=\bar{\rho}.
\end{equation}
\textit{Subcase II ($\rho_l < \rho_r$)}: This can be handled in a similar way. In fact, here we start from $(u_l,\rho_l)$ and reach at $(u^*_{\epsilon},\rho^*_{\epsilon})$ by $S_1$ and from $(u^*_{\epsilon},\rho^*_{\epsilon})$ to $ (u_l,\rho_r)$ by $R_2$. So, the solution is given by :
\begin{equation}
u^{\epsilon}=\left\{
	\begin{array}{llll}
		u_l  & \mbox{if } x <\frac{\rho^*_{\epsilon}u^*_{\epsilon}- \rho_l u_l}{\rho^*_{\epsilon}-\rho_l} t \\
		u^*_{\epsilon} & \mbox{if } \frac{\rho^*_{\epsilon}u^*_{\epsilon}- \rho_l u_l}{\rho^*_{\epsilon}-\rho_l}t< x < \lambda_2(u^*_{\epsilon},\rho^*_{\epsilon})t \\
                    R^{u}_{2}(x/t)(u^*_{\epsilon},\rho^*_{\epsilon}) &\mbox{if } \lambda_2(u^*_{\epsilon},\rho^*_{\epsilon})t < x <\lambda_2(u_l,\rho_r)t   \\
                     u_r                   &\mbox{if }  x>\lambda_2(u_l,\rho_r)t
	\end{array}
\right.
\end{equation}
and

\begin{equation}
\rho^{\epsilon}=\left\{
	\begin{array}{llll}
		\rho_l  & \mbox{if } x <\frac{\rho^*_{\epsilon}u^*_{\epsilon}- \rho_l u_l}{\rho^*_{\epsilon}-\rho_l} t \\
		\rho^*_{\epsilon} & \mbox{if } \frac{\rho^*_{\epsilon}u^*_{\epsilon}- \rho_l u_l}{\rho^*_{\epsilon}-\rho_l}t< x < \lambda_2(u^*_{\epsilon},\rho^*_{\epsilon})t \\
                    R^{\rho}_{2}(x/t)(u^*_{\epsilon},\rho^*_{\epsilon})
                     &\mbox{if } \lambda_2(u^*_{\epsilon},\rho^*_{\epsilon})t < x <\lambda_2(u_l,\rho_r)t   \\
                     \rho_r                   &\mbox{if }  x>\lambda_2(u_l,\rho_r)t
	\end{array}
\right.
\end{equation}
where $R_2(\xi)(\bar{u},\bar{\rho})=(R^{u}_{2}(\xi)(\bar{u},\bar{\rho}),R^{\rho}_{2}(\xi)(\bar{u},\bar{\rho}))$ and $R^{u}_{2}(\xi)(\bar{u},\bar{\rho})$ is obtained by solving 
\begin{equation}
\frac{du}{d\xi}=\sqrt{\frac{{\epsilon}p'(\rho(\xi))}{\rho(\xi)}},\,\,\,\,\, u(0)=\bar{u}
\end{equation}
and
$R^{\rho}_{2}(\xi)(\bar{u},\bar{\rho})$ is obtained by solving 
\begin{equation}
\frac{d\rho}{d\xi}=1, \,\,\,\,\,\,\, \rho(0)=\bar{\rho}.
\end{equation}
Now our aim is to find the limit of $(u^{\epsilon},\rho^{\epsilon})$ as $\epsilon \rightarrow  0$ in both of the above cases.
Since $\rho^*_{\epsilon} \in (\rho_l,\rho_r)$ or $\rho^*_{\epsilon} \in (\rho_r, \rho_l)$ this implies $\rho^*_{\epsilon}$ is bounded. Also $\rho^*_{\epsilon}$ and $u^*_{\epsilon}$ satisfies 
\begin{equation}
(u^*_{\epsilon}-u_l)^2\frac{(\rho^*_{\epsilon}+\rho_r)}{2}=\epsilon(\rho^*_{\epsilon}-\rho_r)(p(\rho^*_{\epsilon})-p(\rho_r))
\end{equation}
Since R.H.S is bounded, as $\epsilon \rightarrow 0$ we get,
\begin{equation}
\lim_{\epsilon\to 0} (u^*_{\epsilon}-u_l)^2\frac{(\rho^*_{\epsilon}+\rho_r)}{2}=0
\end{equation}
that is, $\lim_{\epsilon\to 0}u^*_{\epsilon}=u_l$.
Therefore the solution $(u^{\epsilon},\rho^{\epsilon}) \rightarrow (u,\rho)$ as $\epsilon \rightarrow 0$ where $(u,\rho)$ is given by:
\begin{equation}
\rho=\left\{
	\begin{array}{ll}
		\rho_l  & \mbox{if } x < u_l t \\
		 \rho_r  &\mbox{if }  x>u_l t
	\end{array}
\right.
\end{equation}
and
\begin{equation}
u=\left\{
	\begin{array}{ll}
		u_l  & \mbox{if } x < u_l t \\
		 u_r  &\mbox{if }  x>u_l t
	\end{array}
\right.
\end{equation}
Since here $u_l=u_r$ we have $u \equiv u_l$.\\\\
\textit{Case II $(u_l<u_r)$} : 
The 1st-rarefaction curve passing through $(u_l,\rho_l)$ is given by the solution of the following cauchy problem:
\begin{equation*}
\frac{du}{d\rho}=-\sqrt{\frac{\epsilon p^{'}(\rho)}{\rho}},\;\;\;\;\;\; \rho<\rho_l,\,\,\,\,\,\,\, \,\,\,\,u(\rho_l)=u_l.
\end{equation*}
Note that for this case it does not matter whether $\rho_l<\rho_r$ or $\rho_l>\rho_r$. So W.L.O.G we can take $\rho_l>\rho_r>0$. Now a branch of $R_1$ can be parameterized by a $C^{1}$ function $u_1:[0, \rho_l]\rightarrow [u_l,\infty)$ with a parameter $\rho$. Explicitly $u_1$ can be written as
\begin{equation}
u_1(\rho)-u_l=-\int_{\rho_l}^{\rho} \sqrt{\frac{\epsilon p^{'}(\xi)}{\xi}} d\xi.
\end{equation}
Since $\rho \in [0,\rho_l]$ is bounded and $p$ is increasing, we have $u_{1}(\rho) \rightarrow u_l$ as $\epsilon \rightarrow 0$ decreasingly. Similarly, the 2nd-rarefaction curve is given by the solution of then cauchy problem :
\begin{equation}
\frac{du}{d\rho}=\sqrt{\frac{\epsilon p^{'}(\rho)}{\rho}},\;\;\;\;\;\; \rho<\rho_r,\,\,\,\,\,\,\,\,\,\, u(\rho_r)=u_r
\label{5.20}
\end{equation}
Let $u_2:[0,\rho_r]\rightarrow(-\infty, u_r]$ is a $C^{1}$ function, parameterized branch of $R_2$ satisfying \eqref{5.20} and can be written as
\begin{equation}
u_{2}(\rho)-u_r=\int_{\rho}^{\rho_r} \sqrt{\frac{\epsilon p^{'}(\xi)}{\xi}} d\xi.
\end{equation}
Since $\rho \in [0,\rho_r]$ is bounded and $p$ is increasing, we have $u_2(\rho) \rightarrow u_r$ as $\epsilon \rightarrow 0$ increasingly.
Since $u_l<u_r$, by the above calculation one can see $u_1(0)< u_2(0)$ for small $\epsilon$.
In this case the complete solution is the following:
\begin{equation}
u^{\epsilon}=\left\{
	\begin{array}{lllll}
		u_l  & \mbox{if } x < \lambda_1(u_l,\rho_l)t \\
		R^{u}_{1}(x/t)(u_l,\rho_l) & \mbox{if } \lambda_1(u_l,\rho_l)t< x < \lambda_1(u^{*(1)}_{\epsilon},0)t \\
                    x/t &\mbox{if } \lambda_1(u^{*(1)}_{\epsilon},0)t < x <  \lambda_2(u^{*(2)}_{\epsilon},0)t \\
                     R^{u}_{2}(x/t)(u^{*(2)}_{\epsilon},0) & \mbox{if }  \lambda_2(u^{*(2)}_{\epsilon},0)t <x< \lambda_2(u_r,\rho_r)t  \\
                      u_r                & \mbox{if } x>\lambda_2(u_r,\rho_r)t.
	\end{array}
\right.
\end{equation}
and
\begin{equation}
\rho^{\epsilon}=\left\{
	\begin{array}{lllll}
		\rho_l  & \mbox{if } x < \lambda_1(u_l,\rho_l)t \\
		R^{\rho}_{1}(x/t)(u_l,\rho_l) & \mbox{if } \lambda_1(u_l,\rho_l)t< x < \lambda_1(u^{*(1)}_{\epsilon},0)t \\
                    0 &\mbox{if } \lambda_1(u^{*(1)}_{\epsilon},0)t < x <  \lambda_2(u^{*(2)}_{\epsilon},0)t \\
                     R^{\rho}_{2}(x/t)(u^{*(2)}_{\epsilon},0) & \mbox{if }  \lambda_2(u^{*(2)}_{\epsilon},0)t <x< \lambda_2(u_r,\rho_r)t  \\
                      \rho_r                & \mbox{if } x>\lambda_2(u_r,\rho_r)t.
	\end{array}
\right.
\end{equation}
Where $R^u_1(.)$, $R^{\rho}_1(.)$, $R^u_2(.)$, $R^{\rho}_2(.)$ are defined above.
\\\\
Now it remains to find the limit of $(u^{\epsilon},\rho^{\epsilon})$ as $\epsilon \rightarrow 0$.
Since $u^{*(1)}_{\epsilon}=u_{1}(0)$, we have  $u^{*(1)}_{\epsilon}\rightarrow u_l$ and in the same way $u^{*(2)}_{\epsilon} \rightarrow u_r$ as $\epsilon \rightarrow 0$.
So after passing the limit, we get
\begin{equation}
u(x,t)=\left\{
	\begin{array}{llll}
		u_l  & \mbox{if } x < u_lt \\
		 x/t &\mbox{if } u_lt < x <  u_rt \\
                      u_r                & \mbox{if } x> u_rt 
	\end{array}
\right.
\label{e5.24}
\end{equation}
and
\begin{equation}
\rho(x,t)=\left\{
	\begin{array}{llll}
		\rho_l  & \mbox{if } x < u_lt \\
		 0 &\mbox{if } u_lt < x <  u_rt \\
                      \rho_r                & \mbox{if } x> u_rt .
	\end{array}
\right.
\label{e5.25}
\end{equation}

\begin{remark}
In the solution of $u^{\epsilon}$ we are artificially filling the state $ \lambda_1(u^{*(1)}_{\epsilon},0)t < x <  \lambda_2(u^{*(2)}_{\epsilon},0)t$ by the function $x/t$ as this gives the  least total variation.
\end{remark}
\begin{remark}
The vanishing pressure limit exists  and equal to \eqref{eq3.7}-\eqref{eq3.8} for the case $u_l > u_r$  and   \eqref{e5.24}-\eqref{e5.25} for the case $u_l \leq u_r$ 
with any pressure $p(\rho)$  such that $p(\rho)$  and $p^{\prime}(\rho)$ are increasing. The proofs given here will valid for this case without any change. 
\end{remark}

\section{Another perturbation: adding $\epsilon>0$ in the flux function}
In this short section we propose a different approximation to the system \eqref{e1.1} mentioned earlier in the introduction. The flux function $f(u)=\frac{u^2}{2}$ of the \eqref{e1.1}  of the first function is replaced by $f(u, \epsilon)=\frac{(u+\epsilon)^2}{2}$. This makes the system strictly hyperbolic and is a perturbed system of \eqref{e1.1}.

Let us recall the system \eqref{e1.1}
\begin{equation*}
\begin{aligned}
u_t +(\frac{u^2}{2})_x &=0\\
\rho_t +(\rho u)_x&=0,
\end{aligned}
\end{equation*}
with initial condition
\begin{equation*}
u(x,0)=u_{0}(x),\,\, \rho(x,0)=\rho_0 (x).
\end{equation*}
Perturbed version of the above system can be written as 
\begin{equation}
\begin{aligned}
u_t +(\frac{(u+\epsilon)^2}{2})_x &=0\\
\rho_t +(\rho u)_x&=0.
\end{aligned}
\label{e6.1}
\end{equation}
 with the following Riemann type initial data:
\begin{equation}
\begin{pmatrix}
         u_0 (x)   \\
            \rho_0 (x) \\
         \end{pmatrix}
   =
\begin{cases}
\begin{pmatrix}
         u_l  \\
            \rho_l \\
         \end{pmatrix},\,\,\,\,\,\,\, \textnormal{if}\,\,\,\,\,\,\  x<0\\
 \begin{pmatrix}
         u_r   \\
            \rho_r \\
         \end{pmatrix},\,\,\,\,\,\,\, \textnormal{if}\,\,\,\,\,\,\  x>0.
\end{cases}
\label{e6.2}
\end{equation}
Our aim is to obtain the distributional limit of the solutions $u^{\epsilon}$ and $\rho^{\epsilon}$  of \eqref{e6.1} as $\epsilon$ tends to zero.\\
The eigenvalues and the eigenvectors for the system \eqref{e6.1} are the following:\\
$\lambda_1(u)=u$ and the corresponding eigenvector is $r_1(u)=(0,1)$\\
and\\
$\lambda_2(u)=u+\epsilon$ and the corresponding eigenvector is $r_1(u)=(1,\rho/\epsilon).$\\
Again, $\nabla \lambda_1(u).r_1(u)=0$ and $\nabla \lambda_2(u).r_2(u)=1$. So the first characteristics field is linearly de-generate and the second characteristics field is 
genuinely nonlinear. Let's find explicitly the rarefaction family.\\\\
\textbf{1st-Rarefaction family}: 1st-rarefaction family is the solution of the ODE;
\begin{equation}
\dot{w}(\xi)=r_1(w(\xi)),\,\,\,\,\,  w(\lambda_1(u_l,\rho_l))=(u_l,\rho_l),
\end{equation}
where $w(\xi)=(w_1(\xi),w_2(\xi)).$  So, solving the following pair of ODE:
\begin{equation}
\begin{aligned}
\dot{w_1}(\xi)=0,\,\,\,\,w_1(u_l)=u_l\\
\dot{w_2}(\xi)=1,\,\,\,\,w_2(u_l)=\rho_l,
\end{aligned}
\end{equation}
we get the 1st-rarefaction family as
\begin{equation}
R_1(\xi)=(u_l,\,\,\,\, \xi+\rho_l-u_l).
\end{equation}

\textbf{2nd-Rarefaction family}: 2nd-rarefaction family is the solution of the ODE;
\begin{equation*}
\dot{w}(\xi)=r_2(w(\xi)),\,\,\,\, w(\lambda_2(u_l,\rho_l))=(u_l,\rho_l)
\end{equation*}
where $w(\xi)=(w_1(\xi),w_2(\xi))$\\
That gives the following system of ODEs with initial conditions.
\begin{equation*}
\begin{aligned}
\dot{w_1}(\xi)=1,\,\,\,\,w_1(u_l+\epsilon)=u_l\\
\dot{w_2}(\xi)=\frac{w_2(\xi)}{\epsilon},\,\,\,\,w_2(u_l+\epsilon)=\rho_l
\end{aligned}
\end{equation*}
Solving this pair of ODE, we get the 2nd-Rarefaction family as
\begin{equation}
R_2(\xi)=(\xi-\epsilon, \,\,\,\, \rho_l.\exp(\frac{\xi-(u_l+\epsilon)}{\epsilon})).
\label{e6.6}
\end{equation}
Since the first characteristics field is  linearly degenerate, the 1st-Shock curve and the 1st-Rarefaction curve will coincide, i.e.,  $R_1(\xi)=S_1(\xi)$.\\


\textbf{2nd admissible shock curve}: Second admissible shock curve passing through $( u_l, \rho_l)$ is given by:
\begin{equation}
\rho(u)=\frac{\rho_l (\frac{u-u_l}{2}+\epsilon)}{\epsilon-\frac{u- u_l}{2}},\,\,\, u< u_l, \,\,\, u_l -u  \leq \epsilon.
\end{equation}\\
Main result of this section is the following.

\begin{theorem}
Consider the perturbed Riemann problem (6.3) with the initial data (6.4) such that $u_l<u_r$, then it has a unique weak solution $(u^\epsilon,\rho^\epsilon)$ whose limit as $\epsilon$ $\rightarrow 0 $ is given by 
\begin{equation}
u(x,t)=\left\{
	\begin{array}{llll}
		u_l  & \mbox{if } x < u_lt \\
		 x/t &\mbox{if } u_lt < x <  u_rt \\
                      u_r                & \mbox{if } x> u_rt 
	\end{array}
\right.
\end{equation}
and
\begin{equation}
\rho(x,t)=\left\{
	\begin{array}{llll}
		\rho_l  & \mbox{if } x < u_lt \\
		 0 &\mbox{if } u_lt < x <  u_rt \\
                      \rho_r                & \mbox{if } x> u_rt .
	\end{array}
\right.
\label{e6.9}
\end{equation}
If $u_l \geq  u_r$, then then the limit of the solution $(u^\epsilon,\rho^\epsilon)$  of $(6.3)$   as $\epsilon$ tends to zero is given by

\begin{equation}
(u,\rho)(x,t)= \begin{cases}
	(u_l,\rho_l) \,\,\,\,\, \textnormal{if }\,\,\, u_l t < x <  (\frac{u_l+u_r}{2})t \\
(\frac{u_l + u_r}{2}, \frac{(u_l - u_r)(\rho_l + \rho _r)}{2}\delta_{x=(\frac{u_l+u_r}{2})t } ),\,\,\,\,
\textnormal{if } \,\,\,  x = (\frac{u_l+u_r}{2})t.\\
                    (u_r,\rho_r)\,\,\, \textnormal{if}\,\,\,  x>  (\frac{u_l+u_r}{2})t.
	\end{cases}
\label{e6.10}
\end{equation}
\end{theorem}
\begin{proof}
{\bf{Case 1: $u_l < u_r$}}: 
 The state $(u_l, \rho_l)$ can be joined to $(u_l,\rho^*_\epsilon)$ by 1-shock curve  and $(u_l,\rho^*_\epsilon)$ can be joined to $(u_l, u_r)$ by 2-rarefaction curve.
Then by \eqref{e6.6}, $(u_l,\rho^*_\epsilon)$ will satisfy the following equations.
\begin{equation*}
u_l = \xi-\epsilon,\,\,\,  \rho_r= \rho^*_\epsilon.\exp(\frac{\xi-(u_l+\epsilon)}{\epsilon}).
\end{equation*}
Which yields 

\begin{equation*}
 \xi=u_l + \epsilon,\,\,\,  \rho^*_\epsilon=\rho_r.\exp(\frac{(u_l-u_r)}{\epsilon}).
\end{equation*}


 So the solution for the perturbed problem is given by:
\begin{equation}
(u^{\epsilon},\rho^{\epsilon})(x,t)=\left\{
	\begin{array}{lllll}
		(u_l,\rho_l)  & \mbox{if }\,\,\,\, x < \lambda_1(u_l,\rho_l)t \\
		 (u_l,\rho^*_\epsilon) &\mbox{if }\,\,\, \lambda_1(u_l,\rho_l)t < x < \lambda_2(u_l,\rho^*_\epsilon)t \\
              R_2(\xi)(u_l,\rho_l)                    &\mbox{if }\,\,\,\lambda_2(u_l,\rho^*_\epsilon)< x< \lambda_2(u_r,\rho_r)t \\
              (u_r,\rho_r)       &\mbox{if}\,\,\,\, x>\lambda_2(u_r,\rho_r)t,
	\end{array}
\right.
\end{equation}

where $\lambda_2(R_2(\xi)(u_l,\rho_l)=x/t$, i.e, $\xi=x/t$.

Therefore the solution is given by
\begin{equation}
(u^{\epsilon},\rho^{\epsilon})(x,t)=\left\{
	\begin{array}{lllll}
		(u_l,\rho_l)  & \mbox{if } x < u_lt \\
		 (u_l,\,\,\,\, \rho_r.\exp(\frac{(u_l-u_r)}{\epsilon})) &\mbox{if } u_lt < x <(u_l+\epsilon)t \\
               (x/t-\epsilon, \,\,\,\, \rho^*_\epsilon.\exp(\frac{x/t-(u_r+\epsilon)}{\epsilon}))                    &\mbox{if }(u_l+\epsilon)t< x<(u_r+\epsilon)t \\
              (u_r,\rho_r)       &\mbox{if}\,\,\, x>(u_r+\epsilon)t.
	\end{array}
\right.
\end{equation}
Now as $\epsilon\rightarrow0$ gives the limit \eqref{e6.9} in the sense of distribution.

{\bf{Case 2: $u_l =u_r$}}:
Solutions for the Riemann problem when  $u_l \leq u_r, \,\,\, u_r-u_l  \leq \epsilon$  are given by the following:  The state $(u_l, u_r)$ is connected to $(u_l, \rho^ *)$ by 1st shock family and 
 $(u_l, \rho^ *)$  to  $(u_r, \rho_r)$ by 2nd shock family. Here $ \rho^ *= \frac{\rho_l (\frac{u_r-u_l}{2}+\epsilon)}{\epsilon-\frac{u_r- u_l}{2}}$.
\begin{equation}
(u^{\epsilon},\rho^{\epsilon})(x,t)=\left\{
	\begin{array}{lllll}
		(u_l,\rho_l)  & \mbox{if } x < u_l t \\
		 (u_l,\rho^*) &\mbox{if } u_l t < x <  (\frac{u_l+u_r}{2}+\epsilon)t \\
                    (u_r,\rho_r)   &\mbox{if}\,\,\,  x>  (\frac{u_l+u_r}{2}+\epsilon)t.
	\end{array}
\right.
\label{e6.13}
\end{equation}

{\bf{Case 3: $u_l > u_r$}}:
When   $u_l< u_r, \,\,\, u_r -u_l > \epsilon$, then the solution can not be a  function of bounded variation. We give the solution in the class of  measures by defining $u$ suitably 
along the discontinuity of the first component $u$. In this case the solution $(u, \rho)$  is given by the following.

\begin{equation}
\begin{aligned}
&(u^{\epsilon},\rho^{\epsilon})(x,t)\\&= \begin{cases}
	(u_l,\rho_l) \,\,\,\,\,\,\,\,\,\, \textnormal{if }\,\,\, u_l t < x <  (\frac{u_l+u_r}{2}+\epsilon)t \\
(\frac{u_l + u_r}{2}+\epsilon, \,\,\, \frac{(u_l - u_r)(\rho_l + \rho _r)}{2}+\epsilon(\rho_r - \rho_l)\delta_{x=(\frac{u_l+u_r}{2}+\epsilon)t } ),\,\,\,\,
\textnormal{if }   x = (\frac{u_l+u_r}{2}+\epsilon)t.\\
                    (u_r,\rho_r)\,\,\,\,\,\,\,\, \textnormal{if}\,\,\,  x>  (\frac{u_l+u_r}{2}+\epsilon)t.
	\end{cases}
\end{aligned}
\label{e6.14}
\end{equation}
It can be easily checked that $(u^{\epsilon},\rho^{\epsilon})$ satisfies the equation \eqref{e6.1}-\eqref{e6.2}. The limit of $(u^{\epsilon},\rho^{\epsilon})$ in the equation \eqref{e6.13}-\eqref{e6.14} as $\epsilon$ tends to zero is \eqref{e6.10} in the sense of distribution. This completes the proof.

\end{proof}


\begin{thebibliography}{0}
\bibitem{ba1}
A. Bressan, Hyperbolic Systems of Conservation Laws, The One-Dimensional Cauchy Problem, Oxford University Press, 2005.

\bibitem{Cemp} S. Caprino, R. Esposito, R. Marra and M. Pulvirenti, Hydro-dynamics limits of the Vlasov equation,  Comm. Partial. Diff. Eqs.,
{\bf{18}}(1993) 805-820.

\bibitem{Li1}
G.Q, Chen, H. Liu, Formation of $\delta$-shocks and vaccum states in the vanishing pressure limit of solution to the Euler equation for isentropic fluids,
SIAM J. Math. Anal. {\bf{34}} (2003), no.4, 925-938.

\bibitem{co1}
J.F. Colombeau, New Generalized Functions and Multiplication of Distributions: A Graduate Course, Application to Theoretical and Numerical Solutions of Partial Diﬀerential Equations, Lyon, 1993.


\bibitem{da1}
M. Dafermos, Hyperbolic Conservation Laws in Continuum Physics, Springer, 2016.


\bibitem{di1} DiPerna, Ronald J. Global solutions to a class of nonlinear hyperbolic systems of equations. Comm. Pure Appl. Math. 26 (1973), 1-28.


\bibitem{Ea} S. Earnshaw, On the mathematical theory of sounds, Philos. Trans., {\bf{150}}(1858), 1150-1154. 


 \bibitem{h1}
E. Hopf, The partial differential equation $u_t +u u_x =\mu u_{xx}$,
Comm. Pure Appl. Math., {\bf 13} (1950) 201-230

\bibitem{j2}
K. T. Joseph, A Riemann problem whose viscosity solution
contain $\delta$- measures, Asym. Anal., {\bf7}(1993) 105-120 .

\bibitem{ma2}
 K. T. Joseph, Manas R. Sahoo, Vanishing viscosity approach to a system of conservation laws admitting $\delta^{\prime \prime}$-waves, Commun. Pure Appl. Anal. {\bf{12}} (2013), no. 5, 2091-2118.
 
 
 
 \bibitem{k1}
D.J. Korchinski, Solution of a Riemann problem for a system of conservation laws possessing no classical weak solution, thesis, Adelphi Univers
ity,1977


\bibitem{le1}
P.G. LeFloch, An existence and uniqueness result for two nonstrictly 
hyperbolic systems, in  Nonlinear Evolution Equations that change 
type, (eds) Barbara Lee Keyfitz and Michael Shearer, {\bf IMA} {\bf 27}
(1990) 126-138.


\bibitem{lu1}
 Lu, Yun-Guang Convergence of viscosity solutions to a nonstrictly hyperbolic system. Advances in nonlinear partial differential equations and related areas (Beijing, 1997), 250-266, World Sci. Publ., River Edge, NJ, 1998.
 
 \bibitem{d1}
G.Dal Maso, P.G.LeFloch and F.Murat, Definition and weak stability of 
nonconservative products, J. Math. Pures Appl. {\bf 74} (1995) 483-548.
 
 
\bibitem{n1}
 Darko  Mitrovic, Marko  Nedeljkov,  Delta shock waves as a limit of shock waves, J. Hyperbolic Differ. Equ. 4 (2007), no. 4, 629-653. 
 
 
\bibitem{ob2}
M. Oberguggenberger,  Multiplication of Distributions and Applications to PDEs, Pittman Res. Notes Math., vol. 259, Longman, Harlow, 1992. 

\bibitem{o1}
M. Oberguggenberger, Case study of a nonlinear, nonconservative, nonstrictly hyperbolic system, Nonlinear Anal. 19 (1992) 53-79.

\bibitem{Oe1} K. Oelschlager, On the connection between Hamiltonian many-particle systems and the hydrodynamical equation, Arch. Rat. Mech. Anal., {\bf{115}}(1991), 297-310. 

\bibitem{Oe2} K. Oelschlager, An integro- differential equation modelling a New-tonian dynamics and its scaling limit,  Arch. Rat. Mech. Anal., {\bf{137}}(1997), 99-134.

\bibitem{p1}
E.Yu. Panov, V.M. Shelkovich,
$\delta'$-shock waves as a new type of solutions to system of conservation laws, J. Differential Equations 228 (2006) 49-86.

\bibitem{ma1}
 Manas R. Sahoo, Generalized solution to a system of conservation laws which is not strictly hyperbolic. J. Math. Anal. Appl. {\bf{432}} (2015), no. 1, 214-232

\bibitem{v1}
A.I. Volpert, The space BV and quasi-linear equations,  Math USSR Sb 
{\ bf 2} (1967) 225-267.


\bibitem{Wh}G. B. Whitham, Linear and  Noninear waves, John Wiley and Sons, New York, 1973.

\bibitem{z1}
Ya. Zeldovich, Gravitational instability: an approximate theory for 
large density perturbations, Astron. Astrophys., {\bf 5} (1970) 84-89.

.











 


\
\end{thebibliography}
\end{document}